\documentclass{amsart}
\usepackage{AgCaPreamble}
\usepackage[shortlabels]{enumitem}
\usepackage[]{fullpage}

\theoremstyle{plain}
\newtheorem{theorem}{Theorem}[section]
\newtheorem*{theorem*}{Theorem}
\newtheorem{lemma}[theorem]{Lemma}
\newtheorem{corollary}[theorem]{Corollary}
\newtheorem{prop}[theorem]{Proposition}

\theoremstyle{definition}

\newtheorem{example}[theorem]{Example}

\theoremstyle{remark}

\title{The number of torsion divisors in a strongly \textit{F}-regular ring is bounded by the reciprocal of \textit{F}-signature}
\author{Isaac Martin}
\address{Centre for Mathematical Sciences, University of Cambridge. Wilberforce Rd, Cambridge CB3 0WA}
\email{ikm23@cam.ac.uk}
\date{\today}

\begin{document}

\maketitle

\maketitle

\begin{abstract}
Polstra showed that the cardinality of the torsion subgroup of the divisor class group of a local strongly $F$-regular ring is finite. We  expand upon this result and prove that the reciprocal of the $F$-signature of a local strongly $F$-regular ring $R$ bounds the cardinality of the torsion subgroup of the divisor class group of $R$. 
\end{abstract}

\section{Introduction}
Throughout this article, $R$ is a commutative Noetherian ring of prime characteristic $p > 0$ and $F^e: R\rightarrow R$ to is the $e$th iterate of the Frobenius endomorphism. We also assume that the Frobenius endomophism is a finite map, i.e. that $R$ is $F$-finite. Given an $R$-module $M$, we denote by $F^e_*M$ the $R$-module obtained from $M$ by restricting scalars along the $F^e$. That is, the endofunctor $F^e_*:\Rmod \rightarrow \Rmod$ takes $M$ to the $R$-module $F^e_*M$, which is precisely $M$ as an Abelian group and whose $R$-action is defined according to the $R$-action on $M$ by $r\cdot F^e_*m := r^{pe}m$ (here, if $m \in M$, we use $F^e_*m$ to denote the corresponding element in $F^e_*M$). It is clear that $F^e_*$ is exact. 

%We will be particularly interested in the family of finitely generated $R$-modules $F^e_*R$. The ring of $p^e$th roots of elements in $R$ is often denoted $R^{1/p^e}$, and because the map $\varphi_e: F^e_*R\rightarrow R^{1/p^e}, F^e_*r \mapsto r^{1/p^e}$ is an isomorphism, we use $F^e_R$ where $R^{1/p^e}$ was historically used. 

Associated to $F$-finite local rings is an invariant known as $F$-signature. This was first introduced by Smith and Van den Bergh \cite{smith2002simplicity}, was formally defined by Huneke and Leuschke \cite{Huneke_2002}, and was proven to exist under general hypotheses by Tucker \cite{Tucker_2012}. Because we work only with integral domains, for our purposes we define the $F$-signature of $R$ to be the limit
\begin{equation*}
    s(R) := \lim_{e\rightarrow \infty} \frac{\frk F^e_*R}{\rank_R F^e_*R}.
\end{equation*}
Here, $\frk F^e_*R$ denotes the free-rank of $F^e_*R$, the maximal rank of a free-module appearing in a direct sum decomposition of $F^e_*R$.

The ring $R$ is said to be strongly $F$-regular if for each nonzero $r \in R$ there is some $e \in \bN$ and $\varphi \in \Hom_R(F^e_*R,R)$ such that $\varphi(F^e_*r) = 1$. Aberbach and Leuschke proved that a local ring of prime characteristic is strongly $F$-regular if and only if its $F$-signature is positive \cite{aberbach2002fsignature}. Every stongly $F$-regular ring is a normal domain and therefore has a well-defined divisor class group on $\Spec(R)$, which we call $\Cl(R)$. Polstra showed that if $R$ is strongly $F$-regular, then the torsion subgroup of $\Cl(R)$ is finite \cite{polstra2020theorem}. Together, these results lend plausibility to the following theorem, the primary contribution of this paper:
\begin{theorem*}
    Let $(R,\frakm,k)$ be a local $F$-finite and strongly $F$-regular ring of prime characteristic $p > 0$. Then the cardinality of the torsion subgroup of the divisor class group of $R$ is bounded by $1/s(R)$ where $s(R)$ is the $F$-signature of $R$.
\end{theorem*}
\noindent
The author notes that $1/s(R)$ has previously been used to establish upper bounds on other related invariants, notably on the order of the \'etale fundamental group of a strongly $F$-regular ring \cite{schwede2016} and on the order of an individual torsion divisor $D$ in a strongly $F$-regular ring \cite{carva2017}. These results further motivate this article. We further note that the techniques employed here are largely inspired by the novel proof in \cite[Theorem 3.8]{polstra2019equimultiplicity} of the classic result first proven in \cite{Huneke_2002}: $s(R) = 1$ if and only if $R$ is regular. 

\bigskip

\noindent
\textbf{Acknowledgements.} The author would like to thank Karl Schwede and Thomas Polstra for many valuable hours of mentoring and dialogue. He also thanks Anurag Singh for discussion regarding examples \ref{ex:strict_inequality} and \ref{ex:equality}.

\section{Preliminary Results and Notation}
For $R$-modules $M$ and $N$, denote by $a^{M}(N)$ the maximal number of $M$ summands appearing in a direct sum decomposition of $N$. In the case that $N = F^e_*R$, we say that $a_e^M(R) := a^M(F^e_*R)$. We use $\tors(\Cl(R))$ to denote the torsion subgroup of $\Cl(R)$, the divisor class group of $R$.

\bigskip

\subsection{Divisorial Ideals}
This section is included for convenience, and readers may choose to skip it. The results in this section are not new, but rather a collection of proofs for commonly used tricks involving divisorial ideals. We establish Lemma \ref{lem:1-dim_reg_local_ring_iso} before moving onto the primary result of this section, Proposition \ref{prop:divisor_arithmetic}, which is used throughout this document to manipulate divisorial ideals.

\bigskip

Recall that if $R$ is a Noetherian normal domain with $X = \Spec(R)$, then we let $\Div(X)$ (or sometimes $\Div(R)$) be the free Abelian group on the height 1 primes of $R$. Denote by $K$ the fraction field of $R$. If we fix a height 1 prime $\frakp$ in $R$, then $R_\frakp$ is a regular local ring of Krull dimension 1, and is therefore a principal ideal domain with fraction field $K$. It's maximal ideal is $\frakp R_\frakp$, and is generated by some element $\pi_\frakp \in R_\frakp$. If $0 \neq f \in K$, then we may uniquely write $f$ as $u\pi_\frakp^N$ for some unit $u \in R_\frakp$ and integer $N$. Thus, for each height 1 prime, we have a valuation $\nu_\frakp: K^\times \rightarrow \bZ$ defined 
\begin{equation*}
    \nu_\frakp(f) = N.
\end{equation*}
There are only finitely many height 1 primes $\frakp$ such that $\nu_\frakp(f) \neq 0$, so 
\begin{equation*}
    \div(f) = \sum_{\substack{\frakp \in \Spec(R) \\ \height \frakp = 1}} \nu_\frakp(f)\cdot \frakp
\end{equation*}
is a divisor. We call divisors of the form $\div(f)$ \textit{principal divisors}, and since $\div(f\cdot g) = \div(f) + \div(g)$, the set of principal divisors forms a subgroup in $\Div(R)$. The divisor class group of $R$, denoted $\Cl(R)$, is defined to be the quotient of $\Div(R)$ by this subgroup of principal divisors.

If all the coefficients of the terms in a divisor $D$ are nonnegative, then we say $D$ is \textit{effective} and write $D \geq 0$. Given a Weil divisor $D$, we define the \textit{divisorial ideal} of $D$ to be
\begin{equation*}
    R(D) = \{f \in K^\times ~\mid ~ \div(f) + D \geq 0\} \cup \{0\}.
\end{equation*}
Every divisorial ideal is a finitely generated, rank 1 $R$-module which satisfies Serre's condition $(S_2),$ and conversely, every rank 1 $R$-module which satisfies $(S_2)$ is isomorphic to a divisorial ideal. In particular, this means $R(D)$ is a reflexive module \cite{Hart_94}. We will be particularly interested in how divisorial ideals interact with restriction along Frobenius $F^e_*(-)$, and note here that because $F^e_*$ commutes with $\Hom(-,R)$ it also commutes with the reflexification functor $\Hom_R(\Hom_R(-,R),R) = (-)^{**}$. 

\bigskip

Recall that for a prime $P \in \Spec(R)$, the $n$th symbolic power of $P$ is defined $P^{(n)} = P^nR_P\cap R$. Divisorial ideals can be realized as the intersections of symbolic powers of primes. For a divisor $D = N_1\frakp_1 + ... + N_\ell\frakp_\ell$, 
\begin{equation}\label{eq:symbolic_powers_div}
    R(D) = R(N_1\frakp_1) \cap ... \cap R(N_\ell \frakp_\ell) = \frakp_1^{(-N_1)} \cap ... \cap \frakp_\ell^{(-N_\ell)}.
\end{equation}
Note that if $N \geq 0$ and $\frakp \in \Spec(R)$ is a prime, then
\begin{equation*}
    \frakp^{(-N)} := \{f \in K ~ \mid ~ \nu_\frakp(f) \geq -N\} \cup \{0\}.
\end{equation*}
This means $P^{(-N)}$ consists only of elements in $k$ which have at most an $N$th power of $\pi_P$ in their denominator. We use the following lemma in the proof of Proposition \ref{prop:divisor_arithmetic} (c), and its proof is included for convenience. 

\begin{lemma}\label{lem:1-dim_reg_local_ring_iso}
    Suppose $(R,\frakm)$ is a local principal ideal domain of prime characteristic $p > 0$. Denote by $\langle\pi\rangle$ the maximal ideal $\frakm$. Then for any integers $n,m \in R$, 
    $$F^e_*\langle \pi^n\rangle \otimes_R \langle \pi^m \rangle \cong F^e_*\langle \pi^{n+mp^e}\rangle$$
    via the isomorphism $\varphi: F^e_*x \otimes y \mapsto F^e_* (xy^{p^e})$.
\end{lemma}
\begin{proof}
    We first establish that this map is a $R$-module homomorphism. It is $R$-multiplicative: if $r \in R$, $x \in \langle \pi^n\rangle$ and $y \in \langle \pi^m\rangle$, then
    \begin{align*}
        \varphi\big(r\cdot (F^e_* x \otimes_R y)\big) 
        &= \varphi(F^e_*r^{p^e}x \otimes_R y) \\
        &= F^e_*(r^{p^e}xy^{p^e}) \\
        &= r\cdot F^e_*(xy^{p^e}) = r\cdot \varphi(F^e_*x \otimes_R y),
    \end{align*}
    and by extending additively to arbitrary tensors we have that $\varphi$ is $R$-linear. To see that it is an isomorphism, we define a map
    \begin{equation*}
        \psi:F^e_*\langle \pi^{n + mp^e}\rangle \rightarrow F^e_*\langle \pi^n\rangle \otimes_R \langle \pi^m \rangle, ~ F^e_*(xy^{p^e}) \mapsto F^e_* x \otimes_R y
    \end{equation*}
    Every element of $\langle \pi^{n + mp^e}\rangle = \langle\pi^{mp^e}\cdot \pi^{n}\rangle = \langle \pi^{n}\rangle \cdot \langle \pi^{m}\rangle^{p^e}$ may be realized as a product $x \cdot y^{p^e}$ where $x \in \langle \pi^n\rangle$ and $y \in \langle \pi^m \rangle$, so this map is well-defined and is easily seen to be a morphism of $R$-modules. We then have
    \begin{align*}
        \varphi\circ \psi (F^e_*(xy^{p^e})) = \varphi(F^e_*x \otimes_R y) = F^e_*(xy^{p^e})
    \end{align*}
    and
    \begin{align*}
        \psi \circ \varphi (F^e_*x \otimes_R y) = \psi(F^e_*(xy^{p^e})) = F^e_*x \otimes_R y,
    \end{align*}
    so we conclude that $\varphi$ is an isomorphism.
\end{proof}
We now proceed to the following proposition, which provides a means of manipulating expressions involving tensor products, reflexifications, and scalar-restrictions of divisorial ideals.

\begin{prop}\label{prop:divisor_arithmetic}
    Suppose $(R,\frakm, k)$ is a Noetherian normal domain of prime characteristic $p > 0$. Let $D_1$ and $D_2$ be Weil divisors, and note that $M^{**} = \Hom_R(\Hom_R(M,R),R)$. The following are true:
    \begin{enumerate}[(a)]
        \item $\Hom_R(R(D_1),R(D_2)) \cong R(D_2 - D_1)$
        \item $(R(D_1) \otimes R(D_2))^{**} \cong R(D_1 + D_2)$
        \item $(F^e_*R(D_1) \otimes_R R(D_2))^{**} \cong F^e_*R(D_1 + p^eD_2)$.
    \end{enumerate}
\end{prop}
\begin{proof}
    We first prove $(a)$. Suppose $f \in R(D_2 - D_1)$, and define a map $\varphi_f: R(D_1) \rightarrow K^\times$ by $g \mapsto f\cdot g$. Since $f \in R(D_2 - D_1)$, $\div(f) + D_2 \geq D_1$, and so for any $g \in R(D_1)$,
    \begin{equation*}
        \div(f\cdot g) + D_2 = \div(f) + \div(g) + D_2 \geq \div(g) + D_1 \geq 0,
    \end{equation*}
    hence $\varphi_f(g) = f\cdot g \in R(D_2)$. Each $f \in R(D_2 - D_1)$ therefore defines a map $\varphi_f: R(D_1) \rightarrow R(D_2)$, so $R(D_2 - D_1) \subseteq \Hom_R(R(D_1),R(D_2))$.
    
    Now fix a map $\varphi \in \Hom_R(R(D_1),R(D_2))$. Each divisorial ideal $R(D)$ is rank 1, so tensoring $\varphi: R(D_1) \rightarrow R(D_2)$ gives us a commutative diagram
    \begin{center}
        \begin{tikzcd}
            R(D_1) \arrow[r,"\varphi"] \arrow[d] & R(D_2) \arrow[d] \\
            K \cong R(D_1)\otimes_R K \arrow[r,"\varphi'"] & R(D_2)\otimes_R K \cong K.
        \end{tikzcd}
    \end{center}
    The map $\varphi'$ is linear as a map of $k$-vector spaces, so there is some element $f \in k$ such that $\varphi'(x) = xf$ for every $x \in k$. Tracing through the diagram and using the fact that each divisorial ideal is a submodule of $k$, we realize $\varphi(x) = xf$ as well. This means $R(D_1 - \div(f)) = f\cdot R(D_1) \subseteq R(D_2)$, so $D_1 - \div(f) \leq D_2 \implies D_2 - D_1 + \div(f) \geq 0$, giving us the second inclusion.
    
    \bigskip
    
    Given $(a)$, the proof of $(b)$ follows from the fact that $\Hom(M,-)$ and $- ~ \otimes ~ M$ form an adjoint pair, i.e. that $\Hom(A \otimes B, C) = \Hom(A, \Hom(B,C))$. Indeed, 
    \begin{align*}
        \Hom_R\big(\Hom_R(R(D_1)\otimes R(D_2),R),R\big) &\cong \Hom_R\big(\Hom_R(R(D_1),\Hom(R(D_2),R)), R\big) \\
        &\cong \Hom_R\big(\Hom_R(R(D_1),R(-D_2)),R\big) \\
        &\cong \Hom_R\big(R(-(D_2 + D_1)), R\big) \\
        &\cong R(D_1+D_2).
    \end{align*}
    
    \noindent
    To prove $(c)$, for two divisors $D_1$ and $D_2$ we first notice that the map
    \begin{equation*}
        \varphi: F^e_*R(D_1) \otimes_R R(D_2) \rightarrow F^e_*R(D_1 + p^eD_2), \hspace{1em} F^e_*x \otimes y \mapsto F^e_*(x\cdot y^{p^e})
    \end{equation*}
    is a homomorphism. Indeed, if $x \in R(D_1)$ and $y \in R(D_2)$, then 
    \begin{equation*}
        \div(x\cdot y^{p^e}) + D_1 + p^eD_2 = \div(x) + D_1 + p^e(\div(y) + D_2)) \geq 0,
    \end{equation*}
    so $F^e_*x \otimes y$ lands in $F^e_*R(D_1 + p^eD_2)$. It's $R$-multiplicative: taking $r \in R$, we see
    \begin{equation*}
        \varphi\big(r\cdot (F^e_*x \otimes y)\big) = \varphi\big(F^e_*x \otimes r\cdot y\big) = F^e_*(x\cdot r^{p^e} y^{p^e}) = r\cdot F^e_*(x\cdot y^{p^e}) = r \cdot \varphi\big(F^e_*x \otimes y\big),
    \end{equation*}
    and by extending additive to arbitrary tensors we have that $\varphi$ is $R$-linear. By localizing at some height 1 prime $\frakp \in \Spec(R)$, we get a map
    \begin{equation*}
        \varphi_{\frakp}: F^e_*R(D_1)_\frakp \otimes_{R_\frakp} R(D_2)_\frakp \rightarrow F^e_* R(D_1+p^e D_2)_\frakp 
    \end{equation*}
    where we have taken advantage of the fact $(F^e_*R(D_1) \otimes_R R(D_2))_\frakp \cong F^e_*R(D_1)_\frakp \otimes_{R_\frakp} R(D_2)_\frakp$. We claim $\varphi_\frakp$ is an isomorphism.
    
    Let $n\frakp$ and $m\frakp$ be the components of $\frakp$ in $D_1$ and $D_2$ respectively, where $n$ and $m$ are integers. Because $\frakp$ is height 1, we see $R(D_1)_\frakp \cong \frakp^{-n}R_\frakp \cong \langle \pi^{-n}\rangle$ and $R(D_2)_\frakp \cong \frakp^{-m}R_\frakp \cong \langle \pi^{-m}\rangle$, where $\langle\pi\rangle$ is the maximal ideal $\frakp R_\frakp$ in $R_\frakp$. After localization and composition with the above isomorphisms, the map $\varphi_\frakp$ is defined
    \begin{equation*}
        \varphi_\frakp: F^e_*\langle \pi^{-n}\rangle \otimes_R \langle \pi^{-m} \rangle \rightarrow F^e_*\langle \pi^{-n-mp^e}\rangle, ~ F^e_*x\otimes y \mapsto F^e_* (xy^{p^e}),
    \end{equation*}
    and applying Lemma \ref{lem:1-dim_reg_local_ring_iso} tells us it is an isomorphism.
    
    \bigskip
    
    Since $\frakp$ was chosen arbitrarily, $\varphi$ is an isomorphism after localizing at \textit{any} height $1$ prime. Thus, since $\varphi$ is an isomorphism at the level of height 1 primes, by reflexifying, we see that 
    \begin{equation*}
        (F^e_*R(D_1) \otimes_R R(D_2))^{**} \xrightarrow{\varphi^{**}} \big(F^e_*R(D_1 + p^eD_2)\big)^{**}
    \end{equation*}
    is an isomorphism by \cite[Theorem 1.12]{Hart_94}. Since reflexification commutes with $F^e_*(-)$ and every divisorial ideal is reflexive, $F^e_*R(D_1 + p^eD_2)$ is reflexive as well. This gives us
    \begin{equation*}
        (F^e_*R(D_1) \otimes_R R(D_2))^{**} \cong \big(F^e_*R(D_1 + p^eD_2)\big)^{**} \cong F^e_*R(D_1 + p^eD_2)
    \end{equation*}
    as desired.
\end{proof}

\bigskip

\subsection{Strongly \textit{F}-regular rings}

We now present a refinement of  \cite[Corollary 2.2]{polstra2020theorem}, stated as Lemma \ref{lem:refinement_cor2.2}, which features the same techniques employed by Polstra. We state \cite[Lemma 2.1]{polstra2020theorem} for convenience.
\begin{lemma}[{\cite[Lemma 2.1]{polstra2020theorem}}]\label{lem:pol_lem_2.1}
    Let $(R,\frakm,k)$ be a local normal domain. Let $C$ be a finitely generated $(S_2)$-module, M a rank 1 module, and suppose that $C \cong M^{\oplus a_1} \oplus N_1 \cong M^{\oplus a_2}\oplus N_2$ are choices of direct sum decompositions of $C$ so that $M$ cannot be realized as a direct summand of either $N_1$ or $N_2$. Then $a_1 = a_2$.
\end{lemma}
\begin{lemma}\label{lem:refinement_cor2.2}
    Let $(R,\frakm,k)$ be a local normal domain and $C$ a finitely generated $(S_2)$-module. If $D_1,...,D_t$ are divisors representing distinct elements of the divisor class group and $R(D_i)$ is a direct summand of $C$ for each $1\leq i\leq t$, then 
    $$R(D_1)^{a^{R(D_1)}(C)} \oplus \hdots \oplus R(D_t)^{a^{R(D_t)}(C)}$$ 
    is a direct summand of $C$.
\end{lemma}
\begin{proof}
    Suppose we have found a decomposition
    \begin{equation}\label{eq:initial_decomp}
        C \cong R(D_1)^{n_1} \oplus ... \oplus R(D_t)^{n_t} \oplus N
    \end{equation}
    where $n_i$ is a positive integer for $1\leq i \leq t$. One such decomposition is given by \cite[Corollary 2.2]{polstra2020theorem}, in which each $n_i = 1$. Fix $i \in \{1,...,n\}$. We prove that if $n_i < a^{R(D_i)}(C)$ then $R(D_i)$ must necessarily be a summand of $N$, and in this way, refine $N$ until $$C \cong R(D_1)^{a^{R(D_1)}(C)} \oplus \hdots \oplus R(D_t)^{a^{R(D_t)}(C)} \oplus N.$$
    
    There exists a decomposition of $C$ such that $C \cong R(D_i)^{a^{R(D_i)}(C)} \oplus P$ by definition; hence, it suffices to show that $R(D_i)$ is not a summand of $R(D_1)^{n_1}\oplus...\oplus R(D_t)^{n_t}$ by Lemma \ref{lem:pol_lem_2.1}. We proceed by contradiction and assume instead that $R(D_1)^{n_1}\oplus...\oplus R(D_t)^{n_t}$ has $n_i + 1$ many $R(D_i)$ summands, that is, the $n_i$ summands already present in addition to one extra. Passing our decomposition through $\Hom_R(-,R(D_i))$ and applying Proposition \ref{prop:divisor_arithmetic} means
    \begin{equation*}
        R(D_i - D_1)^{n_1} \oplus ... \oplus R^{\oplus n_i}\oplus ... \oplus R(D_i - D_t)^{n_t}
    \end{equation*}
    necessarily has $n_i+1$ $R$ summands. There then exists a surjective $R$-linear map 
    \begin{equation*}
        R(D_i - D_1)^{n_1} \oplus ... \oplus R^{n_i}\oplus ... \oplus R(D_i - D_t)^{n_t} \rightarrow R^{n_i+1}.
    \end{equation*}
    Quotienting by $R^{\oplus n_i}$ induces a map 
    \begin{equation*}
        \bigoplus_{1 \leq j \leq t, \\ j \neq j} R(D_i - D_j)^{n_j} \rightarrow R,
    \end{equation*}
    and by the locality of $R$ there must be some $j$ not equal to $i$ such that the image of $R(D_i - D_j)$ contains a unit. This means $R(D_i - D_j)$ must have rank free rank 1, but because every divisorial ideal has rank 1, $R(D_i - D_j) \cong R$ as $R$-modules. Thus, $D_i$ and $D_j$ are linearly equivalent, a contradiction. It must then be the case that $R(D_i)$ is a summand of $N$ by Lemma \ref{lem:pol_lem_2.1}.
\end{proof}

It is known that the divisorial ideals of torsion divisors in strongly $F$-regular rings are maximal Cohen-Macaulay modules due to \cite{Patakfalvi_2013} and \cite{dao2016finite}, but we present a novel proof here. If $M$ is a finitely generated module over a local ring $(R,\frakm, k)$ of prime characteristic $p$ and $e \in \bN$, then we let
\begin{equation*}
    I_e(M) = \left\{ \eta \in M \midd \varphi(F^e_* \eta) \in \frakm, \forall \varphi \in \Hom_R(F^e_*M, R) \right\}.
\end{equation*}

\begin{lemma}\label{lem:tor-free-then-free-sum}
    Let $(R,\frakm,k)$ be an $F$-finite strongly $F$-regular ring and $M_i$ a finitely generated torsion free $R$-module for $1 \leq i \leq n$. Then there exists an $e_0 \in \bN$ such that $F_*^{e} M_i$ has a free summand for all $1 \leq i \leq n$ and $e > e_0$.
\end{lemma}
\begin{proof}
    Observe that $F^e_*M$ has a free summand exactly when there is some $\varphi \in \Hom_R(F^e_*M,R)$ such that $\varphi(m) = 1$. To see this, suppose we have such a $\varphi(m) = 1$. If this is the case, then the map $\alpha: R \rightarrow M$ defined $\alpha(1) = m$ is a morphism such that $\varphi \circ \alpha = \id_R$, so the exact sequence 
    \begin{equation*}
        0 \rightarrow \ker \varphi \rightarrow M \rightarrow R \rightarrow 0
    \end{equation*}
    splits and $M \cong \ker \varphi \oplus R$.
    
    Assume that $M$ is a torsion free $R$-module. Lemma 2.3 (4) in \cite{polstra2020theorem} gives us that 
    \begin{equation}
        \bigcap_{e \in \bN} I_e(M) = 0.
    \end{equation}
    Thus, for every $0 \neq \eta \in M$, there is some $e(\eta) \in \bN$ such that $\eta \not\in I_{e(\eta)}(M)$ and therefore some $\varphi \in \Hom_R(F^{e(\eta)}_*M,R)$ such that $\varphi(\eta) \not\in \frakm$. Without loss of generality we take $\varphi(\eta) = 1$. 
    
    Now suppose $M_1,...,M_n$ are torsion free $R$-modules. For each $M_i$, choose $0 \neq \eta_i \in M_i$ and let $e(\eta_i)$ be a natural number depending on $\eta_i$ such that $\eta_i \not\in I_{e(\eta_i)}(M_i)$. Set 
    $$e_0 = \max \{e(\eta_1),...,e(\eta_n)\}.$$ By part (3) of Lemma 2.3 in \cite{polstra2020theorem}, $I_{e_0}(M_i) \subseteq I_{e(\eta_i)}(M_i)$ since $e_0 \geq e(\eta_i)$. Thus, for each $1 \leq i \leq n$ we may find a $\varphi_i \in \Hom_R(F^{e_0}_* M_i, R)$ such that $\varphi_i(\eta_i) = 1$, and conclude that $F^{e_0}_*M_i$ has a free summand for each $1 \leq i \leq n$.
\end{proof}
\begin{prop}\label{prop:tor_div_max_CM}
    Let $(R,\frakm,k)$ be an $F$-finite strongly $F$-regular ring. If $D$ is a torsion divisor, then $R(D)$ is a maximal Macaulay module.
\end{prop}
\begin{proof}
    Since $R(D) \subseteq K$, $R(D)$ is torsion free. Furthermore, since $D$ is a torsion divisor, up to linear equivalence $nD = 0$ for some $0 \neq n \in \bZ$ and the list $\{nD\}_{n \in \bZ}$ is finite. By Lemma (\ref{lem:tor-free-then-free-sum}) there is some $e \in \bN$ such that for all $n \in \bZ$, $F^e_*R(nD)$ has a free summand. This means we may write $F^e_*R(-p^eD) = R \oplus M$ for some module $M$. Tensoring with $R(D)$ and reflexifying yields
    \begin{equation*}
        F^e_*R \cong R(D) \oplus \Hom_R(\Hom_R(M\otimes_R R(D),R),R)
    \end{equation*}
    after applying Proposition \ref{prop:divisor_arithmetic}. Thus, $R(D)$ is a summand of the maximal Cohen-Macaulay $R$-module ($F^e_*R$) so we conclude that $R(D)$ is a maximal Cohen-Macaulay $R$-module.
\end{proof}

\section{Main Result}
Throughout this section, $(R,\frakm,k)$ is a local $F$-finite strongly $F$-regular ring.
\begin{lemma}\label{lem:e0_so_all_D_are_summands}
    Let $D$ be any torsion divisor. There exists an $e_0$ such that if $e \geq e_0$, then $a_e^{R(D)}(R) \geq 1$.
\end{lemma}
\begin{proof}
    Follows immediately from the proof of Proposition \ref{prop:tor_div_max_CM}.
\end{proof}

\begin{lemma}\label{lem:RD_summands_tend_to_Fsig}
    Let $D$ be a torsion divisor. Then 
    \begin{equation*}
        \lim_{e\rightarrow \infty}\frac{a_e^{R(D)}(R)}{\rank F^e_*R} = s(R),
    \end{equation*}
    where $s(R)$ is the $F$-signature of $R$.
\end{lemma}
\begin{proof}
    This proof consists of two parts. We first show that $\frk_R F^e_*R(-p^eD) = a_e^{R(D)}(R)$, and then we calculate the limit.
    
    \bigskip
    
    First $e \in \bN$ and let $n = a_e^{R(D)}(R)$. We have $F^e_*R \cong R(D)^{n}\oplus M$, where $M$ is a finitely generated $R$-module without an $R(D)$ summand. By Proposition \ref{prop:divisor_arithmetic}, applying $- \otimes_R R(-D)$ and then $\Hom_R(\Hom_R(-,R),R)$ to this isomorphism we obtain
    \begin{equation}\label{eqn:1_in_main_lemma}
        F^e_*R(-p^eD) \cong R^n \oplus N
    \end{equation}
    where $N = \Hom_R(\Hom_R(M \otimes_R R(-D),R),R)$. We claim $n = \frk F^e_*R(-p^eD)$. Suppose for the sake of contradiction that $N$ had a free summand, i.e. that $N \cong R \oplus P$ for some $R$-module $P$. Tensoring equation \ref{eqn:1_in_main_lemma} by $R(D)$ and reflexifying gives us
    \begin{equation*}
        F^e_*R \cong R(D)^n \oplus R(D) \oplus \Hom_R(\Hom_R(P\otimes_R R(D),R),R).
    \end{equation*}
    This means $R(D)^{n+1}$ appears as a summand in a direct sum decomposition of $F^e_*R$, which contradicts the maximality of $n$. Thus, $\frk_R F^e_*R(-p^eD) = a_e^{R(D)}(R)$.
    
    \bigskip
    
    For the second part of the proof, we first establish notation. Polstra proved that the torsion subgroup $\tors(\Cl(R))$ of the divisor class group of a strongly $F$-regular ring is finite \cite{polstra2020theorem}, so we may enumerate them: $\tors(\Cl(R)) = \{D_1,...,D_k\}$. We denote the $e$-th term in the sequence defining the $F$-signature of $R(D_i)$ as follows:
    \begin{equation*}
        s_e(R(D_i)) = \frac{\frk F^e_*R(D_i)}{\rank F^e_* R}.
    \end{equation*}
    Since each divisorial ideal is a finitely generated rank 1 module, Tucker tells us \cite[Theorem 4.11]{Tucker_2012}
    \begin{equation*}
        \lim_{e\rightarrow \infty} s_e(R(D_i)) = s(R(D_i)) = s(R)\cdot \rank R(D_i) = s(R)
    \end{equation*}
    for each $1 \leq i \leq k$. In particular, $s_e(R(D_i))$ and $s_e(R(D_j))$ are equivalent Cauchy sequences for each $1 \leq i, j\leq k$. Now set 
    \begin{equation*}
        b_e = \frac{a_e^{R(D)}(R)}{\rank F^e_*R}
    \end{equation*}
    for sake of clarity. We show that the sequence $\{b_e\}$ is equivalent to $\{s_e(R(D_1))\}$ as a Cauchy sequence and conclude that $\lim b_e = s(R)$.
    
    Fix $\epsilon > 0$. By the equivalence of Cauchy sequences, for each $1 \leq i \leq k$, we may find $N_i \in \bN$ such that for all $e \geq N_i$, $|s_e(R(D_i)) - s_e(R(D_j))| < \epsilon$. Notice that since $a_e^{R(D)}(R) = \frk F^e_*R(-p^eD)$ and $-p^eD$ is a torsion divisor, $b_e$ is equal to $s_e(R(D_i))$ for some $1 \leq i \leq k$. If we let $N = \max\{N_1,...,N_k\}$, then for all $e \geq N$, we have
    \begin{equation*}
        |s_e(R(D_1)) - b_e| \leq \max \big\{|s_e(R(D_1)) - s_e(R(D_i))| ~:~ 1 \leq i \leq k \big\} < \epsilon.
    \end{equation*}
    Thus, $\{s_e(R(D_1)) - b_e\}$ is equivalent to the $0$ sequence, so $\{b_e\}$ is equivalent to $\{s_e(R(D_1))\}$ as a Cauchy sequence. We conclude that $\lim_{e\rightarrow \infty} b_e = s(R)$.
\end{proof}
\begin{theorem}\label{thm:main_thm}
    Let $(R,\frakm,k)$ be a local $F$-finite and strongly $F$-regular ring of prime characteristic $p > 0$. Then $$|\tors(\Cl(R))| \leq 1/s(R),$$
    where $\tors(\Cl(R))$ is the torsion subgroup of the divisor class group of $R$.
\end{theorem}
\begin{proof}
    Set
    \begin{equation*}
        n_e = \sum_{D \in \tors(\Cl(R))} a_e^{R(D)}x(R).
    \end{equation*}
    Fix $e_0$ as in Lemma \ref{lem:e0_so_all_D_are_summands}, and let $e \geq e_0$. For each torsion divisor $D$, $R(D)$ is a summand of $F^e_*R$, so by Lemma \ref{lem:refinement_cor2.2} and the fact that $R(D)$ is rank 1 for any torsion divisor, we have that 
    \begin{equation*}
        n_e = \sum_{D \in \tors(\Cl(R))} a_e^{R(D)}(R)\cdot\rank R(D) \leq \rank F^e_*R.
    \end{equation*}
    By Lemma \ref{lem:RD_summands_tend_to_Fsig},
    \begin{align*}
        \lim_{e\rightarrow \infty} \frac{n_e}{\rank F^e_*R} 
        &= \lim_{e\rightarrow \infty} \sum_{D \in \tors(\Cl(R))} \frac{a_e^{R(D)}(R)}{\rank F^e_*R} \\
        &= \sum_{D \in \tors(\Cl(R))} \lim_{e\rightarrow \infty} \frac{a_e^{R(D)}(R)}{\rank F^e_*R} \\
        &= \sum_{D\in \tors(\Cl(R))} s(R) \\
        &= |\tors(\Cl(R))|\cdot s(R).
    \end{align*}
    The limit commutes with the sum since $|\tors(\Cl(R))| < \infty$ by Corollary 3.3 in \cite{polstra2020theorem}. Because $n_e \leq \rank F^e_*R$, 
    \begin{equation*}
        |\tors(\Cl(R))| \cdot s(R) = \lim_{e\rightarrow \infty} \frac{n_e}{\rank F^e_*R}  \leq 1,
    \end{equation*}
    and we conclude 
    \begin{equation*}
        |\tors(\Cl(R))| \leq \frac{1}{s(R)}.
    \end{equation*}
\end{proof}
We immediately obtain the following corollary to Theorem \ref{thm:main_thm}. Local results often extend to graded rings via localization at the maximal ideal, so this is not surprising.
\begin{corollary}
    Let $R$ be a $\bN$-graded $F$-finite and strongly $F$-regular ring of prime characteristic $p > 0$ such that $R_0$ is a field. Then
    \begin{equation*}
        |\tors(\Cl(R))| \leq \frac{1}{s(R)}.
    \end{equation*}
\end{corollary}
\begin{proof}
    Let $\frakm$ denote the unique homogeneous maximal ideal of $R$. Strong $F$-regularity is a local property, so the localization $R_\frakm$ is strongly $F$-regular and therefore $|\tors(\Cl(R_\frakm))| \leq \frac{1}{s(R_\frakm)}$ by Theorem \ref{thm:main_thm}. 
    
    We know $\Cl(R) \rightarrow \Cl(R_\frakm)$ is a bijection by \cite[Corollary 10.3]{fossum} and that $s(R) = s(R_\frakm)$ by \cite[Corollary 6.19]{destefani2018global}, so we have the desired result.
\end{proof}

\subsection{Examples}
Here we provide two examples of graded $F$-finite strongly $F$-regular rings $R$ of prime characteristic $p > 0$ to illustrate that the inequality in \ref{thm:main_thm} is indeed not strict. We note the $F$-signature may be computed via the formula $s(R) = 2 - e_{HK}(R)$ in both examples since $e(R) = 2$, but we opt instead for arguments which do not invoke the Hilbert-Kunz multiplicity. 
\begin{example}\label{ex:strict_inequality}
    Suppose $p>0$ is prime and $R = \frac{\bF_p[w,x,y,z]}{(wx-yz)}$. This is a determinantal ring with $r = s = 2$, in the notation of Singh \cite[Example 3.1]{Anurag_f-sig}, and therefore has dimension $d = r + s - 1 = 3$. By Singh's example, we have that
    \begin{align*}
        s(R) = \frac{1}{d!}\sum_{i=0}^s(-1)^i\binom{d+1}{i}(s-i)^d = \frac{1}{3!}\sum_{i=0}^2(-1)^i\binom{4}{i}(2-i)^3 = \frac{2}{3}.
    \end{align*}
    
    \noindent Since $R$ is a determinant ring satisfying the hypotheses of \cite[7.3.5]{Bruns-Herzog}, we have that $\Cl(R) = \bZ$ and hence $|\tors(\Cl(R))| = 1$, so $|\tors(\Cl(R))| < \frac{1}{s(R)}.$ 
    
    For a local example, let $\frakm = (w,x,y,z)$ and consider $R_\frakm$. By \cite[Corollary 10.3]{fossum} we immediately see $\Cl(R) = \Cl(R_\frakm)$ and by \cite[Corollary 6.19]{destefani2018global} $s(R_\frakm) = s(R)$, so
    \begin{equation*}
        |\tors(\Cl(R_\frakm))| < \frac{1}{s(R_\frakm)}.
    \end{equation*}
\end{example}

\begin{example}\label{ex:equality}
    Suppose $p>0$ is prime and $n \geq 2$ and set $R = \frac{\bF_p[x,y,z]}{xy - z^n}$. The class group of $R$ is isomorphic to $\bZ/n\bZ$ by \cite[Corollary 3.4]{Anurag-divclass}, so it remains to find $s(R)$. Notice that we have the isomorphism $\frac{\bF_p[x,y,z]}{xy - z^n} \cong \bF_p[x^n,xy,y^n]$. The latter ring lends itself well to the calculation of $F$-signature as it is an affine semigroup ring, so we redefine $R = \bF_p[x^n,xy,y^n]$, set $A = \bF_p[x,y]$ and note that $R \subseteq A$.
    
    Fix $e \in \bN$ and set $q = p^e$ and let $\frakm \subseteq A$ denote the homogeneous maximal ideal. By \cite[Lemma 4]{Anurag_f-sig},
    \begin{equation*}
        a_e(R) = \ell\left(\frac{R}{\frakm^{[q]}\cap R}\right).
    \end{equation*}
    Let $S_e$ denote the ring $\frac{R}{\frakm^{[p^e]}\cap R}$. We can form a maximal chain of submodules of $S_e$ entirely from ideals generated by monomials. To see this, let $T$ denote the collection of distinct monomials in $S_e$ and let 
    \begin{equation*}
        (0) = I_0 \subsetneq I_1 \subsetneq ... \subsetneq I_n = S_e
    \end{equation*}
    be a maximal chain of ideals in $S_e$ whose generators are in $T$. Suppose $0\leq i \leq n$, and choose elements $f_1,...,f_m \in T$ so that $(f_1,...,f_m) = I_i$. If $I_{i+1}$ contained two monomials not in $I_i$, then the above chain would not be maximal, so we can find a monomial $f_m \in T$ so that $(f_1,...,f_m,f_{m+1}) = I_{i+1}$. 
    
    Now suppose we have a nonzero coset $\ol{g} \in I_{i+1}/I_{i}$. The representative $g$ must be a nonzero element in $I_{i+1} \setminus I_i$, and therefore $g = a_1f_1 + ... + a_{m+1}f_{m+1}$ with $a_{m+1} \neq 0$. This means the set $\{f_1,...,f_m,g\}$ generates $I_{i+1}$ as an ideal, and so $\langle \ol{g}\rangle = I_{i+1}/I_i$. Since any nonzero element of $I_{i+1}/I_{i}$ generates the entire group, $I_{i+1}/I_{i}$ is simple. This means the above maximal sequence is a composition series, and it therefore suffices to count the number of distinct monomials in $S_e$ to determine $\ell (S_e)$.
    
    The nonzero monomials $x^ay^b$ in $S_e$ are precisely those monomials in $R$ which are not killed by $\frakm^{[q]}$. A monomial $x^ay^b \in R$ must satisfy $x^ay^b = x^{ni}(xy)^jy^{nk} = x^{ni + j}y^{nk + j}$ for some positive integers $i,j$ and $k$, which implies that $a \equiv b \mod (n)$. If $x^ay^b$ is nonzero in $S_e$ then it is not contained in $\frakm^{[q]} = F^e((x,y))A = (x^q,y^q)$ and hence $a < q$ and $b < q$. Likewise, it can be easily seen that any monomial $x^ay^b$ in $A$ for which $a < q$, $b < q$, and $a \equiv b \mod (n)$ is a monomial in $S_e$, hence there is a bijection between the set of distinct monomials in $S_e$ and pairs of nonnegative integers $(a,b)$ satisfying these conditions. 
    
    Suppose for a moment that $q = mn$ for some $m \in \bN$, and fix $a$ so that $0 \leq a \leq q - 1$. The integers congruent to $a$ modulo $n$ are of the form $ni + a$ for some $i\in \bN$, and there are exactly $m$ such distinct integers $b$ such that $0 \leq b \leq q - 1$. As there are $mn$ choices for $a$ and $m$ choices for $b$ given $a$, there are exactly $m^2n$ pairs of integers $(a,b)$ such that $a < q$, $b < q$, and $a \equiv b \mod (n)$. Therefore $a_e(R) = m^2n$.
    
    Now suppose $q$ is once again arbitrary and pick $m_q$ to be the maximal integer such that $m_q n \leq q$. By the special case addressed above we know $m_q^2n \leq a_e(R) \leq (m_q + 1)^2n$. The ring $R$ has Krull dimension 2, therefore $\rank F^e_*R = p^{ed} = q^2$. We have the equality
    \begin{equation*}
        q^2 - (m_qn)^2 = 2q(q - m_qn) - (q - m_qn)^2
    \end{equation*}
    from which we obtain
    \begin{equation*}
        q^2 - (m_qn)^2 \leq 2qn - (q - m_qn)^2 \leq 2qn.
    \end{equation*}
    Using this inequality we see
    \begin{equation*}
        \frac{1}{n} - \frac{m_q^2n}{q^2} \leq \frac{2}{q}
    \end{equation*}
    and
    \begin{equation*}
        \frac{(m_q + 1)^2n}{q^2} - \frac{1}{n} \leq \frac{2qn + n^2 + 2q}{q^2}.
    \end{equation*}
    The rightmost terms in both of the above inequalities approach $0$ as $q \rightarrow \infty$, hence
    \begin{equation*}
        \frac{1}{n} = \lim_{q\rightarrow \infty}  \frac{(m_q + 1)^2n}{q^2} \leq \lim_{q\rightarrow \infty}\frac{a_e(R)}{q^2} \leq \lim_{q\rightarrow \infty}\frac{(m_q + 1)^2n}{q^2} = \frac{1}{n}
    \end{equation*}
    and $\lim_{e\rightarrow \infty}\frac{a_e(R)}{p^{2e}} = \frac{1}{n}$. We conclude that $s(R) = 1/n$ and $|\tors(\Cl(R))| = n$, and in particular, that $$|\tors(\Cl(R))| = 1/s(R).$$
    
    \noindent This result easily localizes as in Example \ref{ex:strict_inequality}.
\end{example}

\bibliographystyle{plain}
\bibliography{bib.bib}

%\printbibliography[title=References]

\end{document}